\documentclass[a4paper,11pt,twoside]{article}
\usepackage{amsmath}
\usepackage{multicol}
\usepackage{amsmath,amsthm}
\usepackage{amsfonts}
\usepackage{amssymb}
\usepackage{dsfont}
\usepackage{hyperref}
\usepackage{color}

\newcommand{\mysection}{\setcounter{equation}{0} \section}

\newtheorem{Definition}{Definition}[section]
\newtheorem{Proposition}{Proposition}[section]
\newtheorem{Lemme}{Lemma}[section]
\newtheorem{Theoreme}{Theorem}

\newtheorem{Remarque}{Remark}[section]
\newtheorem{Example}{Example}[section]

\numberwithin{equation}{section}
\newcommand\R{\mathbb{R}}
\newcommand\vu{\vec{u}}

\title{\bf Blow-up for a nonlinear PDE with fractional Laplacian and singular quadratic nonlinearity}
\author{Diego Chamorro\footnote{Laboratoire de Math\'ematiques et Mod\'elisation d'Evry, Universit\'e d'Evry Val d'Essonne, France.}, \, Elena Issoglio\footnote{School of Mathematics, University of Leeds, UK.} \footnote{Corresponding author}} 

\begin{document}
\maketitle
\begin{scriptsize}
\abstract{We consider an $n$-dimensional parabolic-type PDE with a diffusion given by a fractional Laplacian operator and with a quadratic nonlinearity of the `gradient' of the solution, convoluted with a singular term $\mathfrak{b}$. Our first result is the  well-posedness for this problem: We show existence and uniqueness of a (local in time) mild solution.  The main result is about blow-up of said solution, and in particular we find sufficient conditions on the initial datum  and on the term $\mathfrak{b}$ to ensure  blow-up of the solution  in finite time.}\\

\end{scriptsize}
\mysection{Introduction}\label{Secc_Intro}
In this article we consider the following partial differential equation 
\begin{equation}\label{EquationIntro}
\begin{cases}
\partial_t\mathfrak{u}=-(-\Delta)^{\frac{\alpha}{2}}
\mathfrak{u} +\left((-\Delta)^{\frac{1}{2}}\mathfrak{u}\right)^2\ast \mathfrak{b},\\[4mm]
\mathfrak{u}(0,x)=\mathfrak{u}_0(x), \quad x\in \mathbb{R}^n,
\end{cases}
\end{equation}
where $\mathfrak{u}_0: \mathbb{R}^n\longrightarrow \mathbb{R}$ is a given initial function, the unknown 
$\mathfrak{u}:[0,+\infty[\times \mathbb{R}^n\longrightarrow \mathbb{R}$ is a real-valued function, the term $\mathfrak{b}$ is a real-valued (generalised) function on $ [0,+\infty[\times \mathbb{R}^n$ (in particular it is singular in the space variable $x\in\R^n$ because it belongs to a fractional Sobolev space of low or even negative order), and $(-\Delta)^{\frac{\alpha}{2}}$ with $0<\alpha\leq 2$ is the fractional Laplacian operator (see Section \ref{Secc_Theorems} below for a precise definition of all these objects). Our main objective is to  study well posedness and  blow-up times for solutions to equation \eqref{EquationIntro}  in any dimension $n\geq 1$. \\

Blow-up questions for similar fluid-dynamics equations have been studied in the past. Notably the most important fluid-dynamic equation is  the \emph{Navier-Stokes} equation, given by
\begin{equation}\label{NS}
\partial_t\vu=\Delta \vu-(\vu\cdot \nabla)\vu-\nabla p, \qquad div(\vu)=0,\qquad \vu(0,x)=\vu_0(x),
\end{equation}
where $\vu:[0,+\infty[\times \mathbb{R}^3\longrightarrow \mathbb{R}^3$ is the fluid velocity, $p:[0,+\infty[\times \mathbb{R}^3\longrightarrow \mathbb{R}$ is the pressure of the fluid and $\vu_0:\mathbb{R}^3\longrightarrow \mathbb{R}^3$ is an initial data. Although existence of mild solutions for \eqref{NS} is known in many different functional spaces (see the book \cite{PGLR2}), the study of unique global solutions for big initial data is a very challenging open problem. 
In order to study the possible blow-up for the Navier-Stokes equations, Montgomery-Smith in \cite{Montgomery} proposes a simplified scalar equation, called the \emph{Cheap Navier-Stokes} equation, which reads 
\begin{equation}\label{CheapNS}
\partial_t u=\Delta u+(-\Delta)^{\frac{1}{2}}(u^2), \qquad u(0,x)=u_0,
\end{equation}
where $u:[0,+\infty[\times \mathbb{R}^3\longrightarrow \mathbb{R}$. The link between the cheap and the standard Navier-Stokes equation can be seen by noting that, under mild assumptions, the nonlinearity $(\vu\cdot \nabla)\vu$ can be rewritten in the form $div(\vu\otimes \vu )$. 
From the point of view of regularity, the terms $div(\vu\otimes \vu ) $ and $ (-\Delta)^{\frac12}({u}^2)$ are quite similar (both involving the {\em ``derivative of $u^2$''}), however the action of the operator $(-\Delta)^{\frac12}$ in the Fourier variable is given by the symbol $|\xi|$, which is positive, while the symbol of the derivatives $\partial_j$  for $1\leq j\leq n$ in the  divergence operator is $i\xi_j$ and this symbol conveys sign information that is usually very difficult to exploit. 
In \cite{Montgomery} the author uses Fourier-based tools to  show blow-up in finite time for the solution of \eqref{CheapNS} when the initial data $u_0$ is  suitably chosen. He also finds an explicit upper bound for the blow-up time. Note that a vectorial version of (\ref{CheapNS}) is considered in \cite{Gallagher}, see also \cite{Sinai} or \cite{Tao}. \\

In recent years we also started to see a growing interest in the study of PDEs with singular coefficients (like $\mathfrak b$ in our case), see e.g.\ \cite{hinz_et.al, hinz_zahle, Issoglio13, Issoglio19}  to mention only a few works.  These equations mainly arise from Physics, but may find other applications in different fields too. One possible interpretation of the singular coefficient is to view it as one realisation of some random noise {(hence linking to Stochastic PDEs)}. Depending on the degree of singularity  of the noise, different techniques can be employed to solve SPDEs, the most general being regularity structures \cite{Hairer14, Hairer13} and paracontrolled distributions \cite{gubinelli04, gubinelli-imkeller-perkowski, gubinelli-perkowski15}. 
In the present paper we do not need these sophisticated tools, because the degree of singularity of $\mathfrak b$ is not so bad. Instead we make use of more classical Fourier-based tools together with properties of fractional Sobolev spaces and the action of a suitable semigroup acting on those spaces. \\

Our interest in equation  \eqref{EquationIntro} was sparked by a similar non-linear equation with a distributional coefficient  which was introduced and analysed in  \cite{Issoglio19}. The equation reads
\begin{equation}\label{EquationIntro0}
\begin{cases}
\partial_tu=\Delta u +(\nabla u\cdot \nabla u)  \mathfrak{b},\\[4mm]
u(0,x)=u_0(x),
\end{cases}
\end{equation}
on $[0,T]\times\R^{n}$, where the term $\mathfrak{b}$ is assumed to be singular in the space variable, in particular $\mathfrak{b}\in L^{\infty}([0,T], \mathcal{C}^{-\gamma}(\mathbb{R}^n))$ with $0<\gamma<\frac{1}{2}$ (here the space $\mathcal C^{-\gamma}$ is a Besov space of negative order, hence it includes distributions). 
In  \cite{Issoglio19} the author proved the existence of  a unique mild solution $u:[0,T]\times \mathbb{R}^n\longrightarrow \mathbb{R}$ of \eqref{EquationIntro0} in the space $ \mathcal C^\nu ([0,T], \mathcal C^{2-\gamma+\varepsilon} )$ for some small $\nu, \varepsilon>0$. 
The non-linearity $\nabla u\cdot \nabla u$ imposes some restrictions on the time of existence $T$ of the mild solution, and in particular the unique mild solution of  (\ref{EquationIntro0})  found in \cite{Issoglio19} is local in time, {\em i.e.}, it only exists up to a small time $t_0<T$  (depending on $u_0$) or, equivalently, it exists until time $T$ but it must start with a small enough initial condition $u_0$ (depending on $T$). 
The question on whether the solution of \eqref{EquationIntro0} explodes in finite time is still open. 

Inspired by Montgomery-Smith  \cite{Montgomery}, we replace the nonlinear term $(\nabla u\cdot \nabla u)$ in \eqref{EquationIntro0} by $ ((-\Delta)^{{\frac12}}\mathfrak{u})^2$ in order to be able to apply  Fourier-based tools.  Indeed, from the point of view of regularity the terms $(\nabla u\cdot \nabla u)$ and $ ((-\Delta)^{{\frac12}}\mathfrak{u})^2$ are quite similar (much like in the Navier-Stokes case), but the term $\nabla u\cdot \nabla u$ encodes cancellation properties that are extremely difficult to handle, while   the operator  $(-\Delta)^{{\frac12}}$ has a positive symbol (like in the cheap Navier-Stokes equation \eqref{CheapNS}) which allows the use of Fourier-based tools. We note that  we also  replaced the pointwise product between the  nonlinearity and the term $\mathfrak{b}$ with a convolution. This  makes the equation in some sense smoother, hence one would expect that existence of a solution is easier to obtain while blow-up is harder to obtain.   The main technical advantage though is that the techniques we use here to show blow-up are  Fourier-based, which means that the convolution will become a product at the Fourier level, and the latter is easier to deal with. Finally we replaced the Laplacian with the more general fractional Laplacian.  \\

The plan of the article is the following: In Section \ref{Secc_Theorems},  after recalling some notation and useful facts, we  state our main results of well posedness of  equation \eqref{EquationIntro} in  Theorem \ref{Theo_Existence}, and of its blow-up  in Theorem \ref{Theo_BlowUp}.  Then we present the proof of well-posedness in Section \ref{Secc_Existence}, and of blow-up in Section \ref{Secc_BlowUp}.  We conclude with a short Appendix containing a technical proof.

\section{Notation and main results}\label{Secc_Theorems}
\subsection{Preliminaries}
We give here a precise definition of all the terms of equation (\ref{EquationIntro}) and we start with the diffusion operator $(-\Delta)^{\frac{\alpha}{2}}$ with $0<\alpha\leq 2$. In the particular case when $\alpha=2$, it is the usual Laplacian operator $-\Delta$. If $0<\alpha<2$, the operator $(-\Delta)^{\frac{\alpha}{2}}$ is defined at the Fourier level by the expression 
\begin{equation}\label{Def_FracLaplacian}
\widehat{(-\Delta)^{\frac{\alpha}{2}} \varphi}(\xi)=|\xi|^{\alpha}\widehat{\varphi}(\xi),
\end{equation}
for all functions $\varphi$ in the Schwartz class $\mathcal{S}(\mathbb{R}^n)$ and where $\widehat{\cdot }$ (or $\cdot^{\wedge}$) denotes the Fourier transform. In particular we have 
$$\widehat{(-\Delta)^{\frac{1}{2}} \varphi}(\xi)=|\xi|\widehat{\varphi}(\xi), \qquad \mbox{with }\varphi\in \mathcal{S}(\mathbb{R}^n).$$ 
For $s>0$ real, ]we define the operator $(Id-\Delta)^{\frac{s}{2}}$ by the symbol $(1+|\xi|^2)^{\frac{s}{2}}$, \emph{i.e.}:
\begin{equation*}
\left((Id-\Delta)^\frac{s}{2}\varphi\right)^{\wedge}(\xi) =(1+|\xi|^2)^{\frac{s}{2}} \widehat{\varphi}(\xi),
\end{equation*}
where $\varphi \in \mathcal{S}(\mathbb{R}^n)$. See \cite[Section 6.1 \& Section 6.2.1]{Grafakos2} for further details on these two operators.\\

For $0<\alpha<2$, the semigroup associated to the operator $-(-\Delta)^{\frac{\alpha}{2}}$ will be denoted by $e^{-t(-\Delta)^{\frac{\alpha}{2}}}$ and its action over functions in the Schwartz class $\mathcal{S}(\mathbb{R}^n)$ is given in the Fourier level by 
\begin{equation}\label{PositivitySemigroup}
\left(e^{-t(-\Delta)^{\frac{\alpha}{2}}}\varphi\right)^{\wedge}(\xi)= e^{-t|\xi|^\alpha}\widehat{\varphi}(\xi)=\widehat{\mathfrak{p}^{\alpha}_t}(\xi)\times \widehat{\varphi}(\xi),
\end{equation}
which implies that we have a convolution kernel $\mathfrak{p}^{\alpha}_t$:
\begin{equation}\label{Def_KernelSemigroup}
e^{-t(-\Delta)^{\frac{\alpha}{2}}}(\varphi)=\mathfrak{p}^{\alpha}_t\ast \varphi.
\end{equation}
See the survey paper \cite{Kwasnicki} for more details on the definition of the fractional Laplacian and its corresponding semigroup $e^{-t(-\Delta)^{\frac{\alpha}{2}}}$. See also  \cite[Sections 3.6--3.9]{Jacob}.  We gather in the lemma below some useful results associated to the kernel $\mathfrak{p}^{\alpha}_t$. 
\begin{Lemme}\label{LemmaPropertiesKernel} For $0<\alpha<2$ consider the kernel $\mathfrak{p}^{\alpha}_t$ associated to the operator $-(-\Delta)^{\frac{\alpha}{2}}$. We have the following properties:
\begin{itemize}
\item[(i)] For all $t>0$ we have $\|\mathfrak{p}^{\alpha}_t\|_{L^1}=1$,
\item[(ii)] For all $t>0$ and $s>0$ we have 
$$\|(Id-\Delta)^{\frac{s}{2}}\mathfrak{p}^{\alpha}_t\|_{L^1}\leq C\max\{1, t^{-\frac{s}{\alpha}}\},$$
for some constant $C>0$.
\end{itemize}
Recall that when $\alpha=2$, then the semigroup $e^{t\Delta}$ is the standard heat semigroup, for which all  previous results are also true.
\end{Lemme}
The first point of this lemma follows from the fact that $\mathfrak{p}^\alpha_t$ is given in \cite[formula (7.2), Chapter 7]{Kolokol}, which is a probability density. The second point can also be deduced from the general properties of the symmetric $\alpha$-stable semigroups given in the books \cite{Jacob} and \cite{Kolokol}, but for the sake of completeness, a sketch of the proof of this inequality is given in the appendix.\\ 

Next we introduce Sobolev spaces, for more details  see the book \cite[Chapter 6]{Grafakos2}. 
We define nonhomogeneous Sobolev spaces $H^s(\mathbb{R}^n)$ with $s\in \mathbb{R}$ (see in particular \cite[Definition 6.2.2.]{Grafakos2}) as the set of distributions $\varphi\in \mathcal{S}'(\mathbb{R}^n)$ such that the quantity 
\begin{equation}\label{eq:normH}
\|\varphi\|_{H^s}:=\|(Id-\Delta)^{\frac{s}{2}}\varphi\|_{L^2},
\end{equation} 
 is finite.
The expression in \eqref{eq:normH} defines a norm and the space $H^s(\mathbb{R}^n)$ endowed with this norm is a Banach space. 
If $s>0$ then the norm $\|f\|_{H^s}$ is equivalent to $\|f\|_{L^2}+\|(-\Delta)^{\frac{s}{2}}f\|_{L^2}$ (see \cite[Theorem 6.2.6.]{Grafakos2}). Note that we always have the inequality $\|(-\Delta)^{\frac{s}{2}}f\|_{L^2}\leq \|f\|_{H^s}$. The latter quantity  is actually the  norm in the homogeneous Sobolev space $\dot{H}^s(\mathbb{R}^n) $ and it is denoted by  $\|f\|_{\dot{H}^s} := \|(-\Delta)^{\frac{s}{2}}f\|_{L^2}$ (see \cite[Theorem 6.2.7]{Grafakos2}).
Note also that we have the space identification $H^0(\mathbb{R}^n)=L^2(\mathbb{R}^n)$ and that, for any $0<s_0, s_1$, we have the space inclusions
$$H^{s_0}\subset L^2\subset H^{-s_1}.$$
Note that negative regularity Sobolev spaces $H^{-s}(\mathbb{R}^n)$ can contain objects that are not necessarily functions, in particular if $s>\frac{n}{2}$ then the Dirac mass $\delta_0$ belongs to $H^{-s}(\mathbb{R}^n)$, {see \cite[Example 6.2.3]{Grafakos2}}.\\
    
{\bf Notation}: We will often use the function space $L^\infty([0, T_0], H^1(\mathbb{R}^n))$, which for brevity we sometimes denote by $L^\infty_t H^1_x$.

\subsection{Existence and Uniqueness}

In this paper we are mainly interested in {\em mild solutions} of the problem (\ref{EquationIntro}), which we introduce below. 
\begin{Definition} 
We say that  $\mathfrak{u}\in L^\infty_t H^1_x$ is a mild solution of \eqref{EquationIntro} if it is a solution of  the following integral equation
\begin{equation}\label{EquationFormulationIntegrale0}
\mathfrak{u}(t,x)=e^{-t(-\Delta)^{\frac{\alpha}{2}}}\mathfrak{u}_0(x)+\int_0^te^{-t(-\Delta)^{\frac{\alpha}{2}}}\left( \left((-\Delta)^{\frac{1}{2}}\mathfrak{u}\right)^2\ast \mathfrak{b}\right)(s,x)\, ds.
\end{equation}
\end{Definition}

 Our first main result  deals with the existence of such mild solutions.
\begin{Theoreme}[Existence and uniqueness]\label{Theo_Existence}
Let $(-\Delta)^{\frac{\alpha}{2}}$ be the fractional Laplacian operator with $0<\alpha\leq 2$, and let $\mathfrak{u}_0$ be a given initial data that belongs to the Sobolev space $H^1(\mathbb{R}^n)$. Furthermore, let us assume that we are in one of the following cases:
\begin{description}
\item[Case 1)] Let $1<\alpha\leq 2$ and let $\mathfrak{b} \in L^{\infty}([0,+\infty[, H^{-\gamma}(\mathbb{R}^n))$ with $0\leq \gamma<\alpha-1$. 
\item[Case 2)] Let $0<\alpha\leq 1$ and let $\mathfrak{b}\in L^{\infty}([0,+\infty[, H^{\gamma}(\mathbb{R}^n))$ with $0< 1-\alpha<\gamma<1$.  
\end{description}
 Then there exists a time $T_0>0$ such that the equation (\ref{EquationIntro}) admits a unique mild solution in the space $L^\infty([0, T_0], H^1(\mathbb{R}^n))$.
\end{Theoreme}

The proof of Theorem \ref{Theo_Existence} is postponed to Section \ref{Secc_Existence}. Before moving on, some remarks on the assumptions on $\mathfrak b$ in this theorem are in order.

The singularity of the term $\mathfrak{b}(t,\cdot)$ (in the space variable) is  driven by the parameter $\gamma$, for which we impose a condition related to the smoothness degree $\alpha$ of the fractional Laplacian operator. 
If $1<\alpha\leq 2$ then $\mathfrak{b}(t,\cdot)$ can be quite singular (in fact since the exponent $-\gamma$ is negative, $\mathfrak b$ can be a distribution). On the other hand, if $0<\alpha \leq 1$ then  $\mathfrak{b}(t,\cdot)$ is not allowed to be a singular distribution and in fact it needs to be  Sobolev regular of order $\gamma$ (for some small but positive $\gamma$)  to ensure the existence of a mild solution. 
In some sense the   parameter $\gamma$ must  compensate for the weaker smoothing property of the kernel $\mathfrak{p}_t^\alpha$ when  $0<\alpha\leq 1$. We can see that if the function $\mathfrak{b}(t,\cdot)$ is more singular than what assumed in Theorem \ref{Theo_Existence}, that is if $-\gamma<-(\alpha -1)$ in Case 1) and if $\gamma<1-\alpha$ in Case 2), then the existence of such mild solutions is not  granted by this result. {We believe this is a hard threshold that cannot be overcome by using different techniques, unless one enhances the term $\mathfrak b$ with extra information and uses tools like {\em regularity structures} or {\em paracontrolled distributions}.} 

In terms of function spaces, we do not claim here any kind of optimality. For example, it should be possible to obtain   this existence theorem in a more general framework, by considering Triebel-Lizorkin spaces $F^{s}_{p,q}$ or Besov spaces $B^{s}_{p,q}$  for the space variable. Nevertheless, for the purpose of this article the space $L^\infty_tH^1_x$ is enough.   

\subsection{Blow-up}

Here we will see that under suitable assumptions it is possible to exhibit a blow-up phenomenon in finite time for the mild solution $\mathfrak u$ of equation (\ref{EquationIntro}).  

In order to explicit the blow-up time, we will work with a special initial data $\mathfrak{u}_0$ of the form $\mathfrak{u}_0= A \omega_0$, where $A$ is a positive constant that will be made explicit later and where $\omega_0:\mathbb{R}^n\longrightarrow \mathbb{R}$ is a function defined in the following way: Let $\xi_0\in \mathbb{R}^n$ be given by  $\xi_{0,1}=\xi_{0,2}=\ldots=\xi_{0,n}=\frac{3}{2}$. Then we define the function $\omega_0$ in the Fourier level by the condition:
\begin{equation}\label{eq_omega_0}
\widehat \omega_0 (\xi)= \mathds{1}_{\{|\xi-\xi_0|<\frac{1}{2}\}}.
\end{equation}
We can thus see that, since $\mathfrak{u}_0$ is bounded and compactly  supported in the Fourier variable, then it belongs to all Sobolev spaces $H^s$ for $0\leq s <+\infty$.\\

In Theorem \ref{Theo_BlowUp} below we show that, even though the initial data $\mathfrak{u}_{0}$ is a smooth function, the unique solution  found in  Theorem~\ref{Theo_Existence} blows up in finite time, provided that the initial condition has norm large enough. As before we will decompose our study following the values of the smoothness degree $\alpha$ and the corresponding assumptions on $\mathfrak b$. 
\begin{Theoreme}[Blow-up]\label{Theo_BlowUp}
Let us consider 
 the fractional Laplacian operator  $(-\Delta)^{\frac{\alpha}{2}}$ with $0<\alpha\leq 2$, and let  the initial condition $\mathfrak{u}_0$ be of the form 
\begin{equation}\label{Def_initialData}
\mathfrak{u}_0= A \omega_0,
\end{equation}
where   $\omega_0 $ is given by \eqref{eq_omega_0} and $A$ is a constant such that $A\geq e^{\ln(2)}  2^{5+n} $, where we recall that $n$ is the dimension of the space. For some parameter $\rho\geq 0$ such that $\rho + \alpha\leq 5n +2$, let the term $\mathfrak{b}$ be such that
\begin{equation}\label{ConditionBterm}
C_1(1+|\xi|^2)^{-\frac{\rho}{2}}\leq \widehat{\mathfrak b}(t, \xi),
\end{equation}
where the constant $C_1$ is  independent of $t$ and satisfies the bound 
$$\max\{1,2^{\frac{\rho}{2}-1}\}n^{\frac{\rho+\alpha}2} 2^{10n -1 +\rho+\alpha}\leq C_1<+\infty.$$ 
Furthermore, let us assume that we are in one of the following cases: 
\begin{description}
\item[Case 1)]  Let $1<\alpha\leq 2$ and let $\mathfrak{b} \in L^{\infty}([0,+\infty[, H^{-\gamma}(\mathbb{R}^n))$ with $0\leq \gamma<\alpha-1$. 
\item[Case 2)]  Let $0<\alpha\leq 1$ and let $\mathfrak{b} \in L^{\infty}([0,+\infty[, H^{\gamma}(\mathbb{R}^n))$ with $1-\alpha\leq \gamma<1$. 
\end{description} 
Then the mild solution $\mathfrak{u}$ of (\ref{EquationIntro}) obtained in Theorem \ref{Theo_Existence} blows up at (or before) time $t_*:=\tfrac{\ln (2)}{2^{\alpha}}$, in particular $\|\mathfrak{u}(t_*, \cdot)\|_{H^1_x} =+ \infty$.
\end{Theoreme}

The proof is postponed to Section \ref{Secc_BlowUp}. We observe that, hidden in the assumptions we make, there is a relationship between the dimension
$n$ and the parameters $\rho$ and $\gamma$, as pointed out in the Remark below. 
\begin{Remarque}\label{RmDimension}
If we combine  assumption  \eqref{ConditionBterm} together with the fact that $\mathfrak b(t)$ must belong to a given fractional Sobolev space,  
then we get  a link between the parameters $\gamma, \rho$ and the dimension $n$. In particular:
\begin{description}
\item[Case 1)] Using  \eqref{ConditionBterm} we have 
\begin{align*}
\|\mathfrak b (t, \cdot)\|_{H^{-\gamma}} 
&= \|(Id-\Delta)^{\frac{-\gamma}2} \mathfrak b(t, \cdot) \|_{L^2} = \|(1+|\cdot|^2)^{\frac{-\gamma}2} \widehat{\mathfrak b}(t, \cdot)  \|_{L^2} \\
&\geq  C_1 \|(1+|\cdot|^2)^{\frac{-\gamma}2} (1+|\cdot|^2)^{\frac{-\rho}2} \|_{L^2}  = C_1\left( \int_{\R^n} \frac{1}{(1+|\xi|^2)^{-(\gamma+\rho) }}  d\xi\right)^{1/2}.
\end{align*}
Since we require $\|\mathfrak b (t, \cdot)\|_{H^{-\gamma}} <+\infty$ we must necessarily have $${\int_{\R^n}} \frac{1}{(1+|\xi|^2)^{-(\gamma+\rho) }}  d\xi <+\infty,$$ which is satisfied  only if $2(\gamma+\rho)>n\geq 1$.
\item[Case 2)] 
 In this case   one deduces analogously the condition $2(\rho-\gamma) >n\geq 1$.
\end{description}
 We note that in both cases we can freely choose $\rho$ large enough to reach any desired dimension~$n$, as long as $\rho+\alpha \leq 5n+2$ (i.e.~$\rho\leq 5n +2 -\alpha$ which can always be satisfied together with $\rho > \tfrac n2 -\gamma$ or $\rho > \tfrac n2 +\gamma$, respectively).
\end{Remarque}

We make a few comments on the meaning of the extra assumptions in Theorem \ref{Theo_BlowUp}, in particular on the choice of $\mathfrak u_0$ and on the restriction on $\mathfrak b$. An example of admissible $\mathfrak b$ is given in  Example \ref{ex} below.
\begin{Remarque}
\begin{itemize}
\item[]
\item The assumptions of Theorem \ref{Theo_Existence} are clearly satisfied, so we know that a (local) solution exists. Here we furthermore pick a special $\mathfrak u_0$ and impose an extra condition on the behaviour of $\widehat{\mathfrak b}$ at infinity.
\item The initial condition $\mathfrak u_0 = A \omega_0 $ is actually a smooth function because it belongs to all Sobolev spaces. The key point is that we choose it so that its norm is large enough (because we impose $A$ bigger than some given constant). Note that the constant $A$ is not optimal.
\item The extra condition on $\mathfrak b$ is expressed in terms of a  lower bound on its Fourier transform. The decay at infinity of the Fourier transform of a function/distribution is intimately related to its regularity.
Condition \eqref{ConditionBterm} in fact prohibits too much regularity for the function~$\mathfrak{b}$. This is required to show blow-up, and it amounts to ensure that we are taking an element in $H^{-\gamma}(\R^n)$ or $H^\gamma(\R^n)$ which is actually `singular', and does  not in fact belong to a (much) smoother space.  Note that the constant $C_1$ is not optimal.
\item One could also add an upper bound of the form 
\[\widehat{\mathfrak b}(t, \xi) \leq C_2(1+|\xi|^2)^{-\frac{\rho}{2}} ,\]
with $C_2>C_1$, which is a sufficient condition, together with  $2(\gamma+\rho)>n\geq 1$ and  $2(\gamma-\rho)>n\geq 1$ respectively, to ensure that the element $\mathfrak b$ does belong to the correct fractional Sobolev space. The upper bound on $\widehat{\mathfrak b}$ prevents growth at infinity (in particular the exponent $-\rho/2$ is required to be negative) hence restricting the `irregularity' of $\mathfrak b$. This condition is not necessary, and could be violated pointwise,  but the global behaviour of  $\widehat{\mathfrak b}(t, \xi)$ will be of this form if we are to ensure that  $\mathfrak b$ belongs to the given Sobolev space. See also the second bullet point in  Example \ref{ex} below for more details. 
\end{itemize}
\end{Remarque}

We conclude this section by giving two examples of 
 admissible $\mathfrak b$ that satisfy the hypothesis of Theorem \ref{Theo_Existence} and Theorem \ref{Theo_BlowUp}. 
\begin{Example}\label{ex}
Below we give two examples that are time-homogeneous, $\mathfrak b(t, \cdot) \equiv \mathfrak b(\cdot)$. If one wants a function of time too, it is enough to multiply them by some $f(t)>0$ which is bounded and  with $L^\infty$-norm smaller than or equal to 1. 
\begin{itemize}
\item In dimension $n=1$ and if $\frac32 <\alpha\leq 2$, we can consider $\mathfrak b$ to be the Dirac mass $\delta_0$ (multiplied by a constant $C\geq C_1$, where $C_1$ is given in Theorem \ref{Theo_BlowUp}). Indeed $\delta_0 \in H^{-\gamma}(\R^n) $ if and only if $\gamma > \frac n2$. This corresponds to choosing $\rho =0$. 
\item In dimension $n\geq 1$, let us fix any $C\geq C_1$ and $\rho$ such that $1+\frac n2<\rho\leq 5n$. Then we define $\mathfrak b$ via   its Fourier transform by \[\widehat{\mathfrak b}(\xi) = C (1+|\xi|^2)^{-\rho/2},\]
which formally gives  $\mathfrak b  = C \left ( (1+|\cdot|^2)^{-\rho/2} \right)^\vee$  and $b\in H^{-\gamma}$.
It is often not possible to calculate the explicit expression of $\mathfrak b$, but we know some of its properties. In particular we know that $\mathfrak b$ is a smooth function on $\R^n\setminus\{ 0\}$ and the (exploding) behaviour at 0 is determined by the relationship between $\rho$ and $n$, see \cite[Proposition 6.1.5]{Grafakos2}. Note that these examples are smoother than the Dirac delta, and nevertheless  we still obtain blow-up of the solution. 
\end{itemize}
\end{Example} 

\section{Existence and Uniqueness}\label{Secc_Existence}
In this section we present the proof of existence and uniqueness of a solution  (Theorem \ref{Theo_Existence}). 

\begin{proof}[Proof of Theorem \ref{Theo_Existence}]

The main idea is to apply a Banach contraction principle for quadratic equations in Banach spaces (see the book  \cite[Theorem 5.1]{PGLR2}) 
We will work  in the Banach space $L^\infty_t H^1_x$. To this aim, let us  rewrite Equation (\ref{EquationFormulationIntegrale0}) in the form
$$\mathcal{U}=\mathcal{U}_0+\mathcal{B}(\mathcal{U}, \mathcal{U}),$$
where $\mathcal U$ is $\mathfrak u(t,x)$, the first term on the RHS is given by $\mathcal{U}_0:=e^{-t(-\Delta)^{\frac{\alpha}{2}}}\mathfrak{u}_0(x)$  and the second term is a bilinear application  on $L_t^\infty H_x^1$  given by:
$$\displaystyle{\mathcal{B}(\mathcal{U}, \mathcal{V}):=\int_0^te^{-(t-s)(-\Delta)^{\frac{\alpha}{2}}}\left((-\Delta)^{\frac{1}{2}}\mathcal{U}\right)\left((-\Delta)^{\frac{1}{2}}\mathcal{V}\right)\ast \mathfrak{b}\,ds}.$$ 
We will prove the following estimates for the two terms
\begin{align}
\|\mathcal{U}_0\|_{L^\infty_tH^1_x}&\leq \delta \label{EstimatesIntroFixedPoint-a}\\
\|\mathcal{B}(\mathcal{U}, \mathcal{V})\|_{L^\infty_tH^1_x}&\leq C_{\mathcal{B}}\|\mathcal{U}\|_{L^\infty_tH^1_x}\|\mathcal{V}\|_{L^\infty_tH^1_x}, \label{EstimatesIntroFixedPoint-b}
\end{align}
for some positive constants $\delta, C_{\mathcal B}$.
Once these estimates are in place, we only need to show that 
\begin{equation}\label{RelationShipFixedPoint}
\delta < \frac{1}{4C_\mathcal{B}},
\end{equation}
to conclude that there exists a unique mild solution of (\ref{EquationIntro}) in the space $L^\infty_tH^1_x$ (following  \cite[Theorem 5.1]{PGLR2}).\\

The bounds on the  term  $\mathcal{U}_0(t,\cdot)$ to get  inequality  (\ref{EstimatesIntroFixedPoint-a}) are independent of the paramenter $\alpha$, so will hold for Case 1) and Case 2) of the Theorem.  
By the definition of the semigroup $e^{-t(-\Delta)^{\frac{\alpha}{2}}}$ in (\ref{Def_KernelSemigroup}) and the properties of its associated kernel $\mathfrak{p}^{\alpha}_t$ listed in Lemma \ref{LemmaPropertiesKernel} we have
\begin{eqnarray*}
\|\mathcal{U}_0(t,\cdot)\|_{H^1}&=&\left\|e^{-t(-\Delta)^{\frac{\alpha}{2}}}\mathfrak{u}_0\right\|_{H^1}\\
&\leq& \|\mathfrak{p}^{\alpha}_t\ast(Id-\Delta)^{\frac{1}{2}}\mathfrak{u}_0\|_{L^2}\\
&\leq&\|\mathfrak{p}^{\alpha}_t\|_{L^1}\|(Id-\Delta)^{\frac{1}{2}}\mathfrak{u}_0\|_{L^2}\leq \|\mathfrak{u}_0\|_{H^1},
\end{eqnarray*}
from which we deduce the inequality
\begin{equation}\label{Control01}
\underset{0<t\leq T_0}{\sup}\|\mathcal{U}_0(t,\cdot)\|_{H^1_x}\leq \|\mathfrak{u}_0\|_{H^1} =: \delta,
\end{equation}
and thus the control (\ref{EstimatesIntroFixedPoint-a}) is granted.\\

We now turn our attention to the estimate (\ref{EstimatesIntroFixedPoint-b}), for which a separate proof for each case  is required.

{\bf Case 1): $1<\alpha\leq 2$}. In this case we write (using the definition of Sobolev spaces $H^1$)
\begin{align*}
\|\mathcal{B} (\mathcal{U},  &\mathcal{V})\|_{L^\infty_tH^1_x}\\
= &\sup_{0<t\leq T_0}  \left\|\int_0^te^{-(t-s)(-\Delta)^{\frac{\alpha}{2}}}\left((-\Delta)^{\frac{1}{2}}\mathcal{U}(s,\cdot)\right)\left((-\Delta)^{\frac{1}{2}}\mathcal{V}(s,\cdot)\right) 
\ast \mathfrak{b}(s,\cdot)ds\right\|_{H^1}\\
\leq &\sup_{0<t\leq T_0}  \int_0^t \left\|e^{-(t-s)(-\Delta)^{\frac{\alpha}{2}}}\left((-\Delta)^{\frac{1}{2}}\mathcal{U}(s,\cdot)\right)\left((-\Delta)^{\frac{1}{2}}\mathcal{V}(s,\cdot)\right)  \ast \mathfrak{b}(s,\cdot)\right\|_{H^1} ds\\
=  & \sup_{0<t\leq T_0}  \int_0^t  \left\| (Id-\Delta)^{\frac12}e^{-(t-s)(-\Delta)^{\frac{\alpha}{2}}}\left((-\Delta)^{\frac{1}{2}}\mathcal{U}(s,\cdot)\right)\left((-\Delta)^{\frac{1}{2}}\mathcal{V}(s,\cdot)\right)  \ast \mathfrak{b}(s,\cdot)\right\|_{L^2} ds
\end{align*}
and then by the definition \eqref{Def_KernelSemigroup} for the semigroup,  by properties of the Bessel potential and by using Young inequalities for convolutions, we have  
\begin{align*}
\|\mathcal{B}(\mathcal{U}, \mathcal{V})\|_{L^\infty_tH^1_x}\leq & \underset{0<t\leq T_0}{\sup}\int_0^t\left\|(Id-\Delta)^{\frac{1+\gamma}{2}} \mathfrak{p}^{\alpha}_{t-s}\ast\right.\\
&\quad\left.\left((-\Delta)^{\frac{1}{2}}\mathcal{U}(s,\cdot)\right)\left((-\Delta)^{\frac{1}{2}}\mathcal{V}(s,\cdot)\right)\ast (Id-\Delta)^{-\frac{\gamma}{2}}\mathfrak{b}(s,\cdot)\right\|_{L^2}ds\\
\leq & \underset{0<t\leq T_0}{\sup}\int_0^t\left\|(Id-\Delta)^{\frac{1+\gamma}{2}}  \mathfrak{p}^{\alpha}_{t-s}\right\|_{L^1}\\
&\qquad\qquad\times\left\|\left((-\Delta)^{\frac{1}{2}}\mathcal{U}(s,\cdot)\right)\left((-\Delta)^{\frac{1}{2}}\mathcal{V}(s,\cdot)\right)\right\|_{L^1}\\
&\qquad\qquad\times \left\| (Id-\Delta)^{-\frac{\gamma}{2}}\mathfrak{b}(s,\cdot)\right\|_{L^2}ds. 
\end{align*}
Now by the properties of the kernel $\mathfrak{p}^{\alpha}_{t}$ stated in Lemma \ref{LemmaPropertiesKernel} and recalling that $\mathfrak{b}\in L^\infty_tH^{-\gamma}_x$ we can write
\begin{eqnarray*}
\|\mathcal{B}(\mathcal{U}, \mathcal{V})\|_{L^\infty_tH^1_x}&\leq&  C\underset{0<t\leq T_0}{\sup}\int_0^t\max\{1,\left(t-s\right)^{-\frac{1+\gamma}{\alpha}}\}\\
&&\qquad\times\left\|\left((-\Delta)^{\frac{1}{2}}\mathcal{U}(s,\cdot)\right)\left((-\Delta)^{\frac{1}{2}}\mathcal{V}(s,\cdot)\right)\right\|_{L^1}\left\|\mathfrak{b}(s,\cdot)\right\|_{H^{-\gamma}}ds\\
&\leq & C\|\mathfrak{b}\|_{L^\infty_tH^{-\gamma}_x}\;\underset{0<t\leq T_0}{\sup}\int_0^t\max\{1,\left(t-s\right)^{-\frac{1+\gamma}{\alpha}}\}\\
&&\qquad\times\left\|(-\Delta)^{\frac{1}{2}}\mathcal{U}(s,\cdot)\right\|_{L^2}\left\|(-\Delta)^{\frac{1}{2}}\mathcal{V}(s,\cdot)\right\|_{L^2}ds\\
&\leq& C\|\mathfrak{b}\|_{L^\infty_tH^{-\gamma}_x}\|\mathcal{U}\|_{L^\infty_t H^{1}_x} \|\mathcal{V}\|_{L^\infty_t H^{1}_x}\\
&& \underset{0<t\leq T_0}{\sup}\int_0^t\max\{1,\left(t-s\right)^{-\frac{1+\gamma}{\alpha}}\}ds,
\end{eqnarray*}
where the term $\left(t-s\right)^{-\frac{1+\gamma}{\alpha}}$ is integrable since $0\leq \gamma<\alpha-1$ and $1<\alpha\leq 2$, so we finally obtain
\begin{equation*}
\|\mathcal{B}(\mathcal{U}, \mathcal{V})\|_{L^\infty_tH^1_x}\leq C T_0^{1-\frac{1+\gamma}{\alpha}}\|\mathfrak{b}\|_{L^\infty_tH^{-\gamma}_x}\|\mathcal{U}\|_{L^\infty_t {H}^{1}_x} \|\mathcal{V}\|_{L^\infty_t {H}^{1}_x}
\end{equation*}
which is \eqref{EstimatesIntroFixedPoint-b} with \begin{equation}\label{CB1}
C_{\mathcal B} = C T_0^{1-\frac{1+\gamma}{\alpha}}\|\mathfrak{b}\|_{L^\infty_tH^{-\gamma}_x}.
\end{equation}

{\bf Case 2):}  $0<\alpha\leq 1$. To start with, we proceed similarly as for Case 1) and   we write
\begin{align*}
\|\mathcal{B}(\mathcal{U}, &\mathcal{V})\|_{L^\infty_tH^1_x}\\
&\leq \sup_{0<t\leq T_0} \int_0^t \left\| (Id-\Delta)^{\frac12}e^{-(t-s)(-\Delta)^{\frac{\alpha}{2}}}\left((-\Delta)^{\frac{1}{2}}\mathcal{U}(s,\cdot)\right)\left((-\Delta)^{\frac{1}{2}}\mathcal{V}(s,\cdot)\right)\ast\mathfrak{b}(s,\cdot)\right\|_{L^2} ds\\
&=  \sup_{0<t\leq T_0} \int_0^t \left\| (Id-\Delta)^{\frac12}\mathfrak{p}^\alpha_{t-s} \ast\left((-\Delta)^{\frac{1}{2}}\mathcal{U}(s,\cdot)\right)\left((-\Delta)^{\frac{1}{2}}\mathcal{V}(s,\cdot)\right)\ast\mathfrak{b}(s,\cdot)\right\|_{L^2} ds.
\end{align*} 
Note that, in this case, the regularity of the kernel 
$\mathfrak{p}^{\alpha}_{t-s}$ is critical in the space we are
 working with, since $\|(Id-\Delta)^{\frac{1}{2}}\mathfrak{p}^{\alpha}_{t-s}\|_{L^1}\leq C\max\{1, (t-s)^{-\frac{1}{\alpha}}\}.$ This is the reason why we have to consider 
 $\mathfrak{b}(s,\cdot) \in H^{\gamma}$ for some positive 
$\gamma$, in particular for  $1-\alpha<\gamma<1$. Indeed the idea is similar as Case 1), but here we multiply by $(Id-\Delta)^{\frac\gamma2}$ the term $\mathfrak b(s,\cdot)$ (which makes it belong to a more singular space) and so we can multiply by  $(Id-\Delta)^{-\frac\gamma2}$ the kernel $\mathfrak p^\alpha_{t}$, effectively giving it some more regularity (and integrability). We get
\begin{align*}
\|\mathcal{B}(\mathcal{U}, \mathcal{V})\|_{L^\infty_tH^1_x}
\leq & \underset{0<t\leq T_0}{\sup}\int_0^t\left\|(Id-\Delta)^{\frac{1-\gamma}{2}} \mathfrak{p}^{\alpha}_{t-s}\ast\right.\\
&\left.\left((-\Delta)^{\frac{1}{2}}\mathcal{U}(s,\cdot)\right)\left((-\Delta)^{\frac{1}{2}}\mathcal{V}(s,\cdot)\right)\ast (Id-\Delta)^{\frac{\gamma}{2}}\mathfrak{b}(s,\cdot)\right\|_{L^2}ds\\
\leq & \underset{0<t\leq T_0}{\sup}\int_0^t\left\|(Id-\Delta)^{\frac{1-\gamma}{2}} \mathfrak{p}^{\alpha}_{t-s}\right\|_{L^1}\\
 &\times \left\|\left((-\Delta)^{\frac{1}{2}}\mathcal{U}(s,\cdot)\right)\left((-\Delta)^{\frac{1}{2}}\mathcal{V}(s,\cdot)\right)\right\|_{L^1}\left\|(Id-\Delta)^{\frac{\gamma}{2}}\mathfrak{b}(s,\cdot)\right\|_{L^2}ds\\
\leq &C\underset{0<t\leq T_0}{\sup}\int_0^t\max\{1,\left(t-s\right)^{-\frac{(1-\gamma)}\alpha}\}\\
&\qquad\times\left\|\left((-\Delta)^{\frac{1}{2}}\mathcal{U}(s,\cdot)\right)\left((-\Delta)^{\frac{1}{2}}\mathcal{V}(s,\cdot)\right)\right\|_{L^1}\left\|\mathfrak{b}(s,\cdot)\right\|_{H^{\gamma}}ds,
\end{align*}
having used again Lemma \ref{LemmaPropertiesKernel}. Now since 
since $1-\alpha< \gamma<1$ by assumption, the term $\left(t-s\right)^{-\frac{(1-\gamma)}\alpha}$ is integrable and we obtain
\begin{equation*}
\|\mathcal{B}(\mathcal{U}, \mathcal{V})\|_{L^\infty_tH^1_x}\leq C T_0^{1-\frac{(1-\gamma)}\alpha}\|\mathfrak{b}\|_{L^\infty_tH^{\gamma}_x}\|\mathcal{U}\|_{L^\infty_t\dot{H}^{1}_x} \|\mathcal{V}\|_{L^\infty_t\dot{H}^{1}_x},
\end{equation*}
which is \eqref{EstimatesIntroFixedPoint-b} with 
 \begin{equation}\label{CB2}
 C_\mathcal B = C T_0^{1-\frac{(1-\gamma)}\alpha}\|\mathfrak{b}\|_{L^\infty_tH^{\gamma}_x}.
\end{equation} 
The proof is completed by choosing $T_0$ in \eqref{CB1} and \eqref{CB2} small enough to ensure \eqref{RelationShipFixedPoint}.
\end{proof}

\begin{Remarque}\label{rm:T0}
We observe that the existence (and uniqueness) of the solution is only local in time. Indeed combining the constraint on $\delta := \|\mathfrak{u}_0\|_{H^1}$ given by $\delta<1/(4C_{\mathcal B})$ (see  \eqref{RelationShipFixedPoint})  with the expressions for $C_{\mathcal B}$ in the two cases  \eqref{CB1} and \eqref{CB2}, we have that the initial condition must be small enough depending on $T_0$ and on the term $\mathfrak{b}$, in particular
 $$ \|\mathfrak{u}_0\|_{H^1}<\frac{1}{4C}
\begin{cases}
T_0^{-(1-\frac{1+\gamma}{\alpha})}\|\mathfrak{b}\|^{-1}_{L^\infty_tH^{-\gamma}_x}& \mbox{if } 1<\alpha\leq 2,\\[3mm]
T_0^{-(1-\frac{1-\gamma}{\alpha})}\|\mathfrak{b}\|^{-1}_{L^\infty_tH^{\gamma}_x}& \mbox{if } 0<\alpha\leq1.\\
\end{cases}
$$
Alternatively, one can choose arbitrarily the initial condition, but the time $T_0$ must then be small enough, depending on the size of the initial data $\mathfrak{u}_0$ and to the size of the term $\mathfrak{b}$:
\begin{equation*}
0<T_0<
\begin{cases}
\left(4C\|\mathfrak{u}_0\|_{H^1}\|\mathfrak{b}\|_{L^\infty_tH^{-\gamma}_x}\right)^{1-\frac{1+\gamma}{\alpha}} &  \mbox{if } 1<\alpha\leq 2,\\[3mm]
\left(4C\|\mathfrak{u}_0\|_{H^1}\|\mathfrak{b}\|_{L^\infty_tH^{-\gamma}_x}\right)^{1-\frac{1-\gamma}{\alpha}} &  \mbox{if } 0<\alpha\leq 1.
\end{cases}
\end{equation*}
\end{Remarque}

\section{Blow-up}\label{Secc_BlowUp}

In this Section we investigate the  time of explosion for the solution $\mathfrak{u}$, that is we prove Theorem~\ref{Theo_BlowUp}. To this aim we first state and prove some auxiliary results about certain properties of the solution~$\mathfrak u$.

Let $\mathfrak{u}$ be the unique mild solution to the equation (\ref{EquationIntro}) according  to Theorem \ref{Theo_Existence}. The solution exists   on the time interval $[0,T_0]$, where 
 the size of $T_0$ is related to the size the initial data $\mathfrak{u}_0$, see Remark \ref{rm:T0}. We denote by $T_{\max}$ the maximal time of existence of the solution, which may be infinite. Clearly $T_0\leq T_{\max}$. The first interesting property of the solution of equation (\ref{EquationIntro}) is related to its positivity in the Fourier variable. 

\begin{Proposition}[Positivity]\label{Propo_geq0}
Let the hypotheses of Theorem \ref{Theo_Existence} hold. Moreover let us assume that  $\widehat{\mathfrak{u}}_0(\xi) \geq 0$ and $\widehat{\mathfrak{b}}(t,\xi)\geq 0$. Then the unique mild solution $\mathfrak{u}$ of equation (\ref{EquationIntro}) satisfies $\widehat{\mathfrak{u}}(t,\xi)\geq 0$ for all $t\leq T_{\max}$. 
\end{Proposition}  
We observe that the assumptions in Theorem \ref{Theo_BlowUp} imply  the assumptions in Proposition \ref{Propo_geq0}. 

\begin{proof} 
Using the Picard iteration scheme, we know that the unique mild solution $\mathfrak{u}$ found in Theorem \ref{Theo_Existence} is the limit in $L_t^\infty H^{1}_x$   as $j\to +\infty$ of $\mathfrak{u}_j $, where 
$$\mathfrak{u}_j(t,x) :=e^{-t(-\Delta)^{\frac{\alpha}{2}}} \mathfrak{u}_0(x) + \int_0^te^{-(t-s)(-\Delta)^{\frac{\alpha}{2}}}\left( \left((-\Delta)^{\frac{1}{2}}\mathfrak{u}_{j-1}\right)^2\ast \mathfrak{b}\right)(s,x)\, ds, \quad \text{ for }j\geq1.$$
If we take the Fourier transform in the space variable of $\mathfrak{u}_j$  
and use the identity (\ref{PositivitySemigroup}) we have
\begin{equation*}
\widehat{\mathfrak u}_j(t,\xi)  
= e^{-t|\xi|^\alpha} \widehat{\mathfrak{u}}_0(\xi)
+ \int_0^t e^{-(t-s)|\xi|^\alpha} \left( \left(|\xi|\mathfrak{u}_{j-1}(s,\xi) \ast|\xi|\mathfrak{u}_{j-1} (s,\xi) \right)\widehat{\mathfrak{b}} (s,\xi) \right)\, ds.
\end{equation*}
Since by hypothesis we have $\widehat{\mathfrak{u}}_0(\xi) \geq 0$ and $\widehat{\mathfrak{b}}(t,\xi)\geq 0$, then the positivity of the right-hand side above carries on in the Picard iteration and  the limit $\mathfrak{u}$ satisfies $\widehat{\mathfrak{u}}(t,\xi)\geq 0$.
\end{proof}

We can see from Theorem \ref{Theo_BlowUp} that to show blow-up we need a specially chosen initial condition $\mathfrak u_0 =A \omega_0$, where $\omega_0 $ is defined in \eqref{eq_omega_0}. The proof of the blow-up will be done iteratively, and for this argument we need the following functions $\omega_k:\mathbb{R}^n\longrightarrow \mathbb{R}$  defined iteratively from $\omega_0$ by the condition:
\begin{equation}\label{eq_omega_n}
\widehat \omega_k(\xi) := \widehat \omega_{k-1}(\xi)\ast  \widehat \omega_{k-1}(\xi).
\end{equation}
These functions $\widehat\omega_k$  have  some useful properties which are collected in the following lemma. 
\begin{Lemme}\label{Lema_Omega}
For all $k\geq 0$ we have that 
\begin{itemize}
\item[(i)] the support of the  Fourier transform $\widehat \omega_k$ is contained in the corona 
$$\{\xi \in \mathbb{R}^n:\sqrt n 2^k< |\xi| < \sqrt n  2^{k+1} \},$$
 where $n$ is the dimension of the Euclidean space;
\item[(ii)] the $L^1$-norm of $\widehat\omega_k$ is given by 
$\|\widehat\omega_k\|_{L^1} = \left(\frac{v_n}{2^n}\right)^{2^k}$, where $v_n= \frac{\pi^{\frac{n}{2}}}{\Gamma(\frac{n}{2}+1)}$ is the volume of the $n$-dimensional unit ball.
 \end{itemize}
\end{Lemme}
\begin{proof}
Both properties can be seen by induction. 

\emph{(i)} We will  show by induction a slightly different support property, which implies the support property stated in the Lemma. In particular, we show that the support of $\widehat \omega_k$ is contained in the hypercube 
$$\{\xi\in \R^n : 2^k < \xi_i < 2^{k+1}, \forall i=1,\ldots,n \},$$ 
to which we refer below as `hypercube support property'.  
It is clear that if the support of $\widehat \omega_k$ is contained in the hypercube above, it is also contained in the corona  $\{\xi \in \mathbb{R}^n:\sqrt n 2^k< |\xi| < \sqrt n  2^{k+1} \}$, because the smallest Euclidean norm for $\xi$ in the hypercube is given by $|\xi| = \sqrt{\sum_{i=1}^n \xi_i^2} > \sqrt{\sum_{i=1}^n 2^{2k}} = \sqrt n 2^k $, and similarly for the largest we have $|\xi|< \sqrt n 2^{k+1}$.

Next we prove the hypercube support property.\\
{\em Initial step.} The hypercube support property is clearly true for $k=0$ from the definition of $\omega_0$ given in (\ref{eq_omega_0}) with the specific choice of $\xi_0$. In particular we have that for $\xi \in \text{supp} (\widehat \omega_0) = \{|\xi-\xi_0|<\frac12\}$ we have that each component of $\xi$ is such that $2^0<\xi_i<2^1$. \\
{\em Induction step.}  We will work with $n=1$ in the induction step, because the proof for $n>1$ can be done component-wise.\\ Let $k\geq1$ and assume that the hypercube support property holds for $k-1$, that is,  for $\xi\in\text{supp}(\widehat \omega_{k-1})$ (here $\xi\in\R$) then  $2^{k-1}<\xi<2^{k}$. Let us calculate the support of $\widehat\omega_k$. By definition we have
\begin{align}\label{omega}
\widehat\omega_k(\xi) &= \int_{\R} \widehat \omega_{k-1} (\eta) \widehat \omega_{k-1} (\xi-\eta) d\eta \notag \\
& = \int_\R  \mathds 1_{ \{2^{k-1}<\eta<2^{k}\}} \mathds 1_{\{ 2^{k-1}<\xi-\eta <2^{k}\}} \widehat \omega_{k-1} (\eta) \widehat \omega_{k-1} (\xi-\eta) d\eta.
\end{align} 
It is easy to check that 
\[
 \{2^{k-1}<\eta<2^{k}\} \cap \{ 2^{k-1}<\xi-\eta <2^{k}\} \subseteq  \{2^{k}<\xi<2^{k+1}\} ,
\]
because from the second set we have $2^{k-1}+\eta<\xi  <2^{k}+\eta$ and combining it with the first set we get $2^{k-1}+2^{k-1}<\xi  <2^{k}+2^{k}$. 
Therefore equation \eqref{omega} can be multiplied by $\mathds 1_{ \{2^{k}<\xi<2^{k+1}\}} $ without changing its value. Thus clearly supp$(\widehat \omega_k) \subseteq \{2^{k}<\xi<2^{k+1}\} $ as wanted. 

\emph{(ii)} Here we calculate the $L^1$-norm of $\widehat\omega_k$.
\\
{\em Initial step.} We set $k=0$ and we easily get
$$\|\widehat{\omega}_0\|_{L^1} = \int_{\{|\xi-\xi_0|<\frac{1}{2}\}}d \xi =|B(0, \tfrac{1}{2})|=\frac{v_n}{2^n}.$$\\
{\em Induction step.}  
 By the hypothesis of induction we assume that $\|\widehat\omega_{k-1}\|_{L^1} = \left(\frac{v_n}{2^n}\right)^{2^{k-1}}$, for some $k\geq1$. 
 Note that all  functions $\widehat{\omega}_{k}$ are positive. Then using the hypercube support property from part (i) and the definition of $\widehat{\omega}_{k}$ we have
\begin{align*} 
\| \widehat{\omega}_k\|_{L^1} 
&= \int_{\R^n} | \widehat \omega_k(\xi) |d\xi= \int_{\R^n} \widehat \omega_k(\xi) d\xi\\
 &= \int_{ \{2^{k}<\xi_i<2^{k+1}, \, \forall i\}} \int_{\R^n} \widehat{\omega}_{k-1} (\eta) \widehat{\omega}_{k-1}(\xi-\eta) d\eta \, d\xi\\
 &= \int_{ \{2^{k}<\xi_i<2^{k+1}, \, \forall i\}} \int_{ \{2^{k-1}<\eta_i<2^{k}, \, \forall i\}} \widehat{\omega}_{k-1} (\eta) \widehat{\omega}_{k-1}(\xi-\eta) d\eta \, d\xi\\
 &= \int_{ \{2^{k-1}<\eta_i<2^{k}, \, \forall i\}} \widehat{\omega}_{k-1} (\eta) \int_{ \{2^{k}<\xi_i<2^{k+1}, \, \forall i\}}   \widehat{\omega}_{k-1}(\xi-\eta) d\xi\ d\eta  \\
 &=\int_{ \{2^{k-1}<\eta_i<2^{k}, \, \forall i\}} \widehat{\omega}_{k-1} (\eta) \int_{\R^n}   \widehat{\omega}_{k-1}(\xi-\eta) d\xi\ d\eta, 
\end{align*}
having used the fact that given  $\eta_i\in (2^{k-1},2^k)$ and  $\xi_i-\eta_i\in (2^{k-1},2^k)$, then we automatically have $\xi_i\in (2^{k},2^{k+1})$ for all $i$. Thus the inner integral is the $L^1$-norm of $\widehat \omega_{k-1}$ and we get
\begin{align*}
\| \widehat{\omega}_k\|_{L^1} 
&=\int_{ \{2^{k-1}<\eta_i<2^{k}, \, \forall i\}} \widehat{\omega}_{k-1} (\eta) \|\widehat\omega_{k-1}\|_{L^1} d\eta\\
&= \|\widehat\omega_{k-1}\|_{L^1} \int_{ \R^n} \widehat{\omega}_{k-1} (\eta)  d\eta\\
&=\|\widehat\omega_{k-1}\|_{L^1}^2.
\end{align*}
Then we obtain $\|\widehat\omega_{k}\|_{L^1}= \left(\frac{v_n}{2^n}\right)^{2^{k-1} 2}=\left(\frac{v_n}{2^n}\right)^{2^{k}}$
as wanted. 
\end{proof}

The next result is  a key lower bound for the Fourier transform of the solution $\mathfrak{u}$ of  equation (\ref{EquationIntro}) associated to  initial data $\mathfrak{u}_0$. This lower bound makes use of the functions $\omega_k$ defined above and of another family of functions, $\Phi_k$, given by
\begin{equation}\label{eq_alpha}
\Phi_k(t) := e^{-t 2^{k+\alpha}} 2^{- 5 (2^k-1)}{2^{5nk}}.
\end{equation}

\begin{Proposition}[Lower bound]\label{Propo_LowerBound}
Let the assumptions from Theorem \ref{Theo_BlowUp} hold (in particular $\mathfrak{u}_0 = A \omega_0 $), and let $ \widehat \omega_k$ be defined as in Lemma \ref{Lema_Omega}, starting from $\omega_0 $. Let  $t_*=\tfrac{\ln (2)}{2^{\alpha}}$. 
Then the unique mild solution $\mathfrak{u}$ of equation (\ref{EquationIntro})  verifies the following lower bound for all $k\geq 0$
\begin{equation}\label{eq_lowerbound}
\widehat{\mathfrak{u}}(t,\xi) \geq A^{2^k} \Phi_k(t) \, \widehat \omega_k(\xi),
\end{equation}
for any $t\geq t_*$.
\end{Proposition}
\begin{proof}
In order to prove the inequality (\ref{eq_lowerbound}) we will first derive a general lower bound which will be used later on. Using the mild formulation \eqref{EquationFormulationIntegrale0} and recalling the fact that $\widehat{\mathfrak{p}^\alpha_t}(\xi) =  e^{-t|\xi|^\alpha}$ by identity (\ref{PositivitySemigroup}), and the fact that $\widehat{ (-\Delta)^{\frac12} \mathfrak{u}}(t,\xi) = |\xi| \widehat{\mathfrak{u}}(t,\xi)$, we have for all $t\geq0$, and in particular for all $t\geq t_*$, that
\begin{eqnarray}\label{FirstStep0LowerBound}
\widehat{\mathfrak{u}}(t,\xi)&=&\widehat{\mathfrak{p}^\alpha_t}(\xi)\widehat{\mathfrak{u}}_0(\xi)+\int_0^t\widehat{\mathfrak{p}^\alpha_{t-s}}(\xi)\left(|\xi| \widehat{\mathfrak{u}}(s, \xi)\ast  |\xi| \widehat{\mathfrak{u}}(s, \xi)\right) \widehat{\mathfrak{b}}(s,\xi) \, ds\notag\\
&= & e^{-t|\xi|^\alpha}\widehat{\mathfrak{u}}_0(\xi) +\int_0^t e^{-(t-s)|\xi|^\alpha} \left(|\xi| \widehat{\mathfrak{u}}(s, \xi)\ast  |\xi| \widehat{\mathfrak{u}}(s, \xi)\right) \widehat{\mathfrak{b}}(s,\xi)\, ds.
\end{eqnarray}
We now proceed  to show \eqref{eq_lowerbound} by induction.\\
{\em Initial Step.} We set $k=0$. Note that by assumption we have $\widehat{\mathfrak b}(t, \cdot)\geq 0$ and since we have $\widehat{\mathfrak{u}}_0(\xi) = A \widehat{\omega}_0(\xi)\geq 0$, so by the positivity property stated in Proposition \ref{Propo_geq0}, we have $\widehat{\mathfrak{u}}(t,\xi)\geq 0$ for all $t\geq 0$ and thus all the terms inside the integral on the right-hand side of (\ref{FirstStep0LowerBound}) are positive. Thus we can write
$\widehat{\mathfrak{u}}(t,\xi)\geq e^{-t|\xi|^\alpha}\widehat{\mathfrak{u}}_0(\xi).$
Now we use the definition of $\Phi_0$ given in (\ref{eq_alpha}) and of $\omega_0$ given in (\ref{eq_omega_0}),  $\widehat{\mathfrak u}_0 = A \widehat\omega_0 $, together with the fact that  supp$(\widehat\omega_0)\subset\{1<|\xi|<2\}$ to get
\begin{align*}
\widehat{\mathfrak{u}}(t, \xi) 
&\geq e^{-t |\xi|^\alpha} \widehat{\mathfrak{u}}_0(\xi)=  e^{-t |\xi|^\alpha}  A\widehat \omega_0 (\xi)\\
&\geq A e^{-t 2^\alpha}  \widehat \omega_0(\xi) = A \Phi_0(t) \, \widehat{\omega}_0(\xi),
\end{align*}
which is (\ref{eq_lowerbound}) for $k=0$.
\\
{\em Induction step.} Let $k\geq 1$. Consider $t\geq t_*$ and assume that the inequality \eqref{eq_lowerbound} holds for  $k-1$, that is  $\widehat{\mathfrak{u}}(t,\xi) \geq A^{2^{k-1}} \Phi_{k-1}(t) \, \widehat \omega_{k-1}(\xi)$ . Since we have $\widehat{\mathfrak{u}}_0(\xi)\geq 0$, by the lower bound \eqref{FirstStep0LowerBound} we get
\begin{eqnarray}
\widehat{\mathfrak{u}}(t, \xi) 
&\geq&  \int_0^t e^{-(t-s)|\xi|^\alpha}  \left(|\xi| \widehat{\mathfrak{u}}(s, \xi)\ast  |\xi| \widehat{\mathfrak{u}}(s, \xi)\right)  \widehat{\mathfrak{b}}(s,\xi) \, ds\notag \\
&\geq&  \int_0^t e^{-(t-s)|\xi|^\alpha}  \left(|\xi| A^{2^{k-1}} \Phi_{k-1}(s) \, \widehat \omega_{k-1}(\xi) \ast  |\xi| A^{2^{k-1}} \Phi_{k-1}(s) \, \widehat \omega_{k-1}(\xi)\right) \widehat{\mathfrak{b}}(s,\xi) \, ds\notag\\
& \geq &A^{2^{k}}\int_0^t e^{-(t-s)|\xi|^\alpha} \Phi_{k-1}^2(s)  \left(|\xi|  \widehat{\omega}_{k-1}(\xi) \ast  |\xi|  \widehat{\omega}_{k-1}(\xi)  \right) \widehat{\mathfrak{b}}(s,\xi)\, ds. \notag \\ \label{EstimateBeforeB}
& \geq & C_1 A^{2^{k}}\int_0^t e^{-(t-s)|\xi|^\alpha} \Phi_{k-1}^2(s)\left(|\xi|  \widehat{\omega}_{k-1}(\xi) \ast  |\xi|  \widehat{\omega}_{k-1}(\xi) \right)  (1+|\xi|^2)^{-\frac{\rho}{2}} ds,
\end{eqnarray}
having used in the last inequality the lower bound $ C_1(1+|\xi|^2)^{-\frac{\rho}{2}}\leq \widehat{\mathfrak b}(t, \xi) $ assumed in  \eqref{ConditionBterm}. 
Recall now that the support of the functions $\widehat{\omega}_{k-1}(\xi)$ is contained in $ \{\sqrt n 2^{k-1}<|\xi|< \sqrt n 2^{k}\}$ by Lemma \ref{Lema_Omega} part (i), and in particular one has $2^{k-1}<|\xi|$ if $\xi\in\text{supp}(\widehat\omega_{k-1})$. Using this bound and the expression  $\widehat{\omega}_{k}=\widehat{\omega}_{k-1}\ast \widehat{\omega}_{k-1}$ we have
\begin{eqnarray*}
|\xi|  \widehat{\omega}_{k-1}(\xi) \ast  |\xi|  \widehat{\omega}_{k-1}(\xi)&=&\int_{\mathbb{R}^n}|\xi-\eta|\widehat{\omega}_{k-1}(\xi-\eta) |\eta|\widehat{\omega}_{k-1}(\eta)d\eta\\
&> &\int_{\mathbb{R}^n}  2^{k-1} \widehat{\omega}_{k-1}(\xi-\eta)  2^{k-1}\widehat{\omega}_{k-1}(\eta) d\eta=  2^{2(k-1)}\widehat{\omega}_{k}(\xi), 
\end{eqnarray*}
and thus from \eqref{EstimateBeforeB} we obtain
\begin{equation}\label{EqLowerBoundInductive1}
\widehat{\mathfrak{u}}(t, \xi) \geq C_1A^{2^{k}}\int_0^t e^{-(t-s)|\xi|^\alpha} \Phi_{k-1}^2(s) 2^{2(k-1)}\widehat{\omega}_{k}(\xi) (1+|\xi|^2)^{-\frac{\rho}{2}} ds. 
\end{equation}
Thanks to the support property from Lemma \ref{Lema_Omega}, for $\xi \in \text{supp}(\widehat \omega_k)$  we have $ |\xi|< \sqrt n 2^{k+1}$.  Since by assumption $\rho\geq 0$, we get 
$$(1+|\xi|^2)^{\frac \rho 2} \leq  \max\{1,2^{\frac{\rho}{2}-1}\} (1+|\xi|^\rho)  \leq  \max\{1,2^{\frac{\rho}{2}-1}\} (1+ n^{\frac\rho 2} 2^{(k+1)\rho})  \leq  \max\{1,2^{\frac{\rho}{2}-1}\} n^{\frac\rho2} 2^{(k+1)\rho+1}  .$$ 
Thus we can write 
\begin{align}\label{InequalityInside1}
\widehat{\omega}_{k}(\xi)(1+|\xi|^2)^{-\frac{\rho}{2}} 
& \geq  \widehat{\omega}_{k}(\xi)\frac{ n^{-\frac\rho2} 2^{-(k+1)\rho -1}}{  \max\{1,2^{\frac{\rho}{2}-1}\}}  ,
\end{align}
and plugging this lower bound into \eqref{EqLowerBoundInductive1} together  with the explicit expression for $\Phi_k$ given in \eqref{eq_alpha}, we obtain 
\begin{align}
&\widehat{\mathfrak{u}}(t, \xi) \geq C_1A^{2^{k}}\int_0^t e^{-(t-s)|\xi|^\alpha} \Phi_{k-1}^2(s) 2^{2(k-1)}   \widehat{\omega}_{k}(\xi) \frac{n^{-\frac\rho2}  2^{ -(k+1)\rho-1}}{ \max\{1,2^{\frac{\rho}{2}-1}\}}   ds\nonumber \\
&\geq   \frac{C_1 n^{-\frac\rho2}}{ \max\{1,2^{\frac{\rho}{2}-1}\} }    2^{-1}  2^{2(k-1)-\rho(k+1)}  A^{2^{k}} \int_0^t e^{-(t-s)|\xi|^\alpha} \widehat{\omega}_{k}(\xi)  \Phi_{k-1}^2(s) ds\nonumber\\ 
&\geq \frac{C_1 n^{-\frac\rho2}}{  \max\{1,2^{\frac{\rho}{2}-1}\}}    2^{-1}  2^{2(k-1)-\rho(k+1)}  A^{2^{k}} \int_0^t e^{-(t-s)|\xi|^\alpha} \widehat{\omega}_{k}(\xi) \left(   e^{-s 2^{{(k-1)}+\alpha}} 2^{- 5 (2^{{(k-1)}}-1)}  2^{5n(k-1)} \right)^2  ds \nonumber\\
&\geq \frac{C_1 n^{-\frac\rho2}}{ \max\{1,2^{\frac{\rho}{2}-1}\}} 2^{-1}  2^{2(k-1)-\rho(k+1)}  A^{2^{k}}  e^{-2 t 2^{{(k-1)}+\alpha}} 2^{- 10 (2^{{(k-1)}}-1)} 2^{10n(k-1)}\notag\\
& \qquad \times \int_0^t e^{-(t-s)|\xi|^\alpha} \widehat{\omega}_{k}(\xi) ds. \label{eq_int1} 
\end{align}
We now use again  the support property for $\widehat \omega_k$,  so that in the integral above we have $|\xi| <\sqrt n 2^{k+1}$ and the integral can be bounded from below by 
\[
\int_0^t e^{-(t-s)|\xi|^\alpha} \widehat{\omega}_{k}(\xi)     ds
 \geq \int_0^t e^{-(t-s) n^{\frac \alpha2} 2^{\alpha(k+1)}} \widehat{\omega}_{k}(\xi) ds 
= \widehat{\omega}_{k}(\xi)  n^{-\frac \alpha2}  2^{-\alpha(k+1)}  (1- e^{-t n^{\frac \alpha2} 2^{\alpha(k+1)}}).
\]
At this point we observe that, thanks to the choice of $t_*=\tfrac{\ln (2)}{2^{\alpha}}$ and since  $\alpha>0$ we have that  
$$(1- e^{-t n^{{\frac{\alpha}{2}}}2^{\alpha(k+1)}}) \geq (1- e^{-t n^{{\frac{\alpha}{2}}}2^{\alpha}}) \geq (1- e^{-t 2^{\alpha }}) \geq (1- e^{-t_* 2^{\alpha }}) = \tfrac12,$$
 for all $t\geq t_*$ and for all  $k\geq0$, so that the integral above is in fact bounded by 
\[
\int_0^t e^{-(t-s)|\xi|^\alpha} \widehat{\omega}_{k}(\xi)     ds
 \geq \widehat{\omega}_{k}(\xi)  n^{-\frac \alpha2}  2^{-\alpha(k+1)}  2^{-1}.
\]
Plugging this into \eqref{eq_int1} and doing some algebra we get
\begin{align*}
&\widehat{\mathfrak{u}}(t, \xi)\\
&\geq  \frac{C_1 n^{-\frac\rho2}}{ \max\{1,2^{\frac{\rho}{2}-1}\}}    2^{-1}  2^{2(k-1)-\rho(k+1)}  A^{2^{k}}  e^{-2 t 2^{{(k-1)}+\alpha}} 2^{- 10 (2^{(k-1)}-1)}  2^{10n(k-1)} \widehat{\omega}_{k}(\xi)  n^{-\frac \alpha2}  2^{-\alpha(k+1)} 2^{-1}\\
&= \frac{C_1 n^{-\frac{\rho+\alpha}2}}{ \max\{1,2^{\frac{\rho}{2}-1}\}}   2^{-1}  2^{(k+1)(2-\rho-\alpha)} 2^{10n(k-1)} A^{2^{k}}  e^{- t 2^{k+\alpha}} 2^{- 5 (2^{{k}}-1)} \widehat{\omega}_{k}(\xi)\\
&= \frac{C_1 n^{-\frac{\rho+\alpha}2}}{\max\{1,2^{\frac{\rho}{2}-1}\}}  2^{-1-10n+2-\rho-\alpha} A^{2^{k}}  e^{- t 2^{k+\alpha}} 2^{- 5 (2^{{k}}-1)} 2^{k(10n+2-\rho-\alpha)} \widehat{\omega}_{k}(\xi).
\end{align*}
We remark now that by hypothesis we have $0\leq \rho+\alpha\leq 5n+2$ thus we have $10n + 2-\rho-\alpha\geq 5n$ and we obtain 
\begin{align*}
\widehat{\mathfrak{u}}(t, \xi)&\geq \frac{C_1 n^{-\frac{\rho+\alpha}2}}{ \max\{1,2^{\frac{\rho}{2}-1}\}}  2^{1-10n-\rho-\alpha} A^{2^{k}}  e^{- t 2^{k+\alpha}} 2^{- 5 (2^{{k}}-1)} 2^{5nk} \widehat{\omega}_{k}(\xi)\\
&\geq A^{2^{k}}  \Phi_k(t) \widehat{\omega}_{k}(\xi),
\end{align*}
where in the last line we used the hypothesis that $C_1\geq n^{\frac{\rho+\alpha}2}2^{10n-1+\rho+\alpha}\max\{1,2^{\frac{\rho}{2}-1}\}$.
\end{proof}

Using the tools and results above, we can now prove blow-up, that is we can prove Theorem~\ref{Theo_BlowUp}. 

\begin{proof}[Proof of Theorem \ref{Theo_BlowUp}]
We prove Case 1) and Case 2) together because the specific values of the parameters $\alpha$ and $\gamma$ do not play a role here. \\
Let  $t=t_*$. First note that $\|\mathfrak{u}(t_*, \cdot)\|_{H^1_x} \geq \|\mathfrak{u}(t_*, \cdot)\|_{\dot H^1_x} $. Therefore it is enough to show that the $\dot H^1_x$-norm explodes at $t_*$. Using the definition of the $\dot H^1_x$-norm and the Plancherel theorem we have
\begin{align*}
\|\mathfrak{u}(t_*,\cdot)\|^2_{\dot H^1_x} & = \|(-\Delta)^\frac12 \mathfrak{u}(t_*, \cdot)\|^2_{L^2}= \int_{\mathbb{R}^n} |\xi|^2 |\widehat{\mathfrak{u}}(t_*, \xi)|^2 \mathrm d\xi\\
&\geq \sum_{k=0}^{+\infty}  \int_{\{\sqrt n 2^k <|\xi|< \sqrt n 2^{k+1}\}} |\xi|^2 |\widehat{\mathfrak{u}}(t_*, \xi)|^2 \mathrm d\xi
 \end{align*}
and by Proposition \ref{Propo_LowerBound} and the definition of the functions $\Phi_k$ we have
\begin{align}
\nonumber
\|\mathfrak{u}(t_*,\cdot)\|^2_{\dot H^1_x} &  \geq \sum_{k=0}^{+\infty}  \int_{\{\sqrt n 2^k <|\xi|< \sqrt n 2^{k+1}\}} |\xi|^2 A^{2^{k+1}} \Phi_k^2(t_*) \widehat{\omega}_k^2 ( \xi) \mathrm d\xi\\ \nonumber
&  \geq \sum_{k=0}^{+\infty}  \int_{\{\sqrt n 2^k <|\xi|< \sqrt n 2^{k+1}\}} |\xi|^2 A^{2^{k+1}} e^{-t_* 2^{k+\alpha+1}} 2^{- 10 (2^k-1)}  2^{10nk} \widehat{\omega}_k^2 ( \xi) \mathrm d\xi\\
&  \geq \sum_{k=0}^{+\infty} n 2^{2k} A^{2^{k+1}} e^{-t_* 2^{k+\alpha+1}} 2^{- 10 (2^k-1)}  2^{10nk} \int_{\{\sqrt n 2^k <|\xi|< \sqrt n 2^{k+1}\}}  \widehat{\omega}_k^2 ( \xi) \mathrm d\xi.\label{eq_u*} 
 \end{align} 
Next we look at the integral part only. We see that since supp$(\widehat{\omega}_k)\subset \{ \sqrt n 2^k <|\xi|<\sqrt n 2^{k+1}\}$ and since $\widehat{\omega}_k\geq 0$  we have
\[
\int_{\{\sqrt n 2^k <|\xi|< \sqrt n 2^{k+1}\}} \widehat{\omega}_k^2 ( \xi) \mathrm d\xi = \|\widehat{\omega}_k \|_{L^2}^2.
\] 
To find a lower bound for $\|\widehat{\omega}_k \|_{L^2}^2$, let us denote by  $\mathcal{C}(\sqrt n 2^k, \sqrt n 2^{k+1})$ the dyadic corona given by the set $\{ \sqrt n 2^k <|\xi|< \sqrt n 2^{k+1}\}$. Then the volume of the corona is given by 
$$|\mathcal{C}(\sqrt n 2^k, \sqrt n 2^{k+1})| = v_n ((\sqrt n 2^{k+1})^n -(\sqrt n 2^{k})^n )  = v_n n^{\frac n 2 } (2^n-1) 2^{nk} = C(n) 2^{nk},$$
where the constant $C(n):= v_n n^{\frac n 2 } (2^n-1)$ is independent of $k$. Now by H\"older's inequality we obtain
\[
\|\widehat{\omega}_k\|_{L^1} \leq |\mathcal{C}(\sqrt n 2^k, \sqrt n 2^{k+1})|^{\tfrac12} \|\widehat{\omega}_k \|_{L^2} 
\leq C(n)^{\frac{1}{2}}2^{\frac{nk}{2}}\|\widehat{\omega}_k \|_{L^2}, 
\]
and recalling that we have the identity $\|\widehat\omega_k\|_{L^1} = \left(\frac{v_n}{2^n}\right)^{2^k}$, stated in Lemma \ref{Lema_Omega}, then we can write
\[
\|\widehat{\omega}_k\|_{L^2}^2 
\geq C(n)^{-1}2^{-nk} \|\widehat{\omega}_k\|_{L^1}^2 
=    C(n)^{-1}2^{-nk} \left(\frac{v_n}{2^n}\right)^{2^{k+1}}
=    C(n)^{-1}2^{-nk} v_n^{2^{k+1}} 2^{-n 2^{k+1} }.
\]
 Plugging this into \eqref{eq_u*} we obtain
\begin{align}\label{eq_sum} \nonumber
\|\mathfrak{u}(t_*,\cdot)\|^2_{\dot H^1_x} 
&  \geq \sum_{k=0}^{+\infty} n 2^{2k} A^{2^{k+1}} e^{-t_* 2^{k+\alpha+1}} 2^{- 10 (2^k-1)} 2^{10nk}  C(n)^{-1}2^{-nk} v_n^{2^{k+1}} 2^{-n 2^{k+1} } \\ \nonumber
&  =  n 2^{10} C(n)^{-1} \sum_{k=0}^{+\infty}  A^{2^{k+1}} e^{-t_* 2^{k+\alpha+1}} 2 ^{-10 \cdot 2^k}  2^{-n2^{k+1}} \times v_n^{2^{k+1}}   2^{(10n+2)k}2^{-nk} \\
& = n 2^{10} C(n)^{-1}  \sum_{k=0}^{+\infty} \left( \frac{ A^{2}}{ e^{t_* 2^{\alpha+1}}   2^{10+2n}}\right)^{2^k}\times  v_n^{2^{k+1}}   2^{k(9n+2)}.
\end{align} 
Since $v_n^{2^{k+1}}   2^{k(9n+2)}\geq 1$  then \eqref{eq_sum} becomes
\begin{align*}
\|\mathfrak{u}(t_*,\cdot)\|^2_{\dot H^1_x}  
&\geq  n 2^{10} C(n)^{-1}   \sum_{k=0}^{+\infty} \left( \frac{ A^{2}}{ e^{t_* 2^{\alpha+1}}   2^{10+2n}}\right)^{2^k}.
\end{align*} 
A sufficient condition for the latter series to diverge is $\frac{ A^{2}}{ e^{t_* 2^{\alpha+1}}   2^{10+2n}} \geq 1 $. Setting $A= e^{\ln(2)}  2^{5+n} $ and substituting   $t_*=\frac{\ln(2)}{2^{\alpha}}$ one has 
\[
\frac{ A^{2}}{  e^{t_* 2^{\alpha+1}}   2^{10+2n}}  
= \frac{ (e^{\ln(2)}  2^{5+n} )^2  }{ e^{\frac{\ln(2)}{2^{\alpha}}2^{\alpha+1}}  2^{10+2n}}  
 = \frac{ e^{2\ln(2)}  2^{10+2n} }{ e^{2\ln(2)}  2^{10+2n}} 
 = 1,
\]
so any value $A\geq  e^{\ln(2)}  2^{5+n} $ will make the series diverge. Hence the norm $\|\mathfrak{u}(t_*,\cdot)\|_{H^1_x}$ will explode too. 
\end{proof}
 We observe that our proof of blow-up does not work if $A$ is chosen  small enough so that the series converges. This can be easily seen by formula \eqref{eq_sum}.

\section*{Appendix}
We sketch here a proof for the second point of Lemma \ref{LemmaPropertiesKernel} (see also \cite{Kolokol1} for general $\alpha$-stable laws) and for simplicity we only study the estimate
\begin{equation}\label{FracEstimate1}
\|(-\Delta)^\frac{s}{2}\mathfrak{p}^\alpha_t\|_{L^1}\leq Ct^{-\frac{s}{\alpha}}.
\end{equation}
Indeed, by definition we have $\left((-\Delta)^\frac{s}{2}\mathfrak{p}^\alpha_t\right)^{\wedge}(\xi)=|\xi|^se^{-t|\xi|^\alpha}=t^{-\frac{s}{\alpha}}\left(|t^{\frac{1}{\alpha}}\xi|^{s}e^{-|t^{\frac{1}{\alpha}}\xi|^\alpha}\right)$, 
since this quantity is a function that belongs to $L^1$ in the $\xi$ variable, we can apply the inverse Fourier transform to obtain
\begin{eqnarray*}
(-\Delta)^\frac{s}{2}\mathfrak{p}^\alpha_t(x)&=&t^{-\frac{s}{\alpha}}\frac{1}{(2\pi)^n}\int_{\mathbb{R}^n}\left(|t^{\frac{1}{\alpha}}\xi|^{s}e^{-|t^{\frac{1}{\alpha}}\xi|^\alpha}\right)e^{ix\xi}d\xi\\
&=&t^{-\frac{s}{\alpha}}t^{-\frac{n}{\alpha}}\frac{1}{(2\pi)^n}\int_{\mathbb{R}^n}\left(|u|^{s}e^{-|u|^\alpha}\right)e^{i(t^{-\frac{1}{\alpha}}x)u}du=t^{-\frac{s}{\alpha}}t^{-\frac{n}{\alpha}}\left((-\Delta)^{\frac{s}{2}}\mathfrak{p}^\alpha_1\right)(t^{-\frac{1}{\alpha}}x),
\end{eqnarray*}
thus, taking the $L^1$-norm we have the homogeneity identity $\|(-\Delta)^\frac{s}{2}\mathfrak{p}^\alpha_t\|_{L^1}=t^{-\frac{s}{\alpha}}\|(-\Delta)^\frac{s}{2}\mathfrak{p}^\alpha_1\|_{L^1}$, and thus we only need to prove that $\|(-\Delta)^\frac{s}{2}\mathfrak{p}^\alpha_1\|_{L^1}<+\infty$. 

For this we recall the Riemann-Liouville representation of the operator $(-\Delta)^\frac{s}{2}$ (which can be seen by passing to the Fourier level):
$$(-\Delta)^\frac{s}{2}(\mathfrak{p}^\alpha_1)=\frac{1}{\Gamma(k-s/2)}\int_0^{+\infty}\tau^{k-s/2-1}(-\Delta)^{k} (h_\tau\ast\mathfrak{p}^\alpha_1)d\tau,$$
where $h_\tau$ is the standard heat kernel, $\Gamma$ is the usual Gamma function and $k$ is any integer such that $k>s/2$. Then, taking the $L^1$-norm and since $\|h_\tau\|_{L^1}=\|\mathfrak{p}^\alpha_1\|_{L^1}=1$, we have:
\begin{align*}
&\|(-\Delta)^\frac{s}{2}(\mathfrak{p}^\alpha_1)\|_{L^1} \\
&\leq \frac{1}{\Gamma(k-s/2)}\left(\int_0^{1}\tau^{k-s/2-1}\|h_\tau\|_{L^1}\|(-\Delta)^{k} \mathfrak{p}^\alpha_1\|_{L^1}d\tau\right.\\
& \quad +\left.\int_1^{+\infty}\tau^{k-s/2-1}\|(-\Delta)^{k} h_\tau\|_{L^1}\|\mathfrak{p}^\alpha_1\|_{L^1}d\tau\right)\\
&\leq  \frac{1}{\Gamma(k-s/2)}\left(\int_0^{1}\tau^{k-s/2-1}\|(-\Delta)^{k} \mathfrak{p}^\alpha_1\|_{L^1}d\tau+\int_1^{+\infty}\tau^{k-s/2-1}\|(-\Delta)^{k} h_\tau\|_{L^1}d\tau\right)\\
&\leq  C\|(-\Delta)^{k} \mathfrak{p}^\alpha_1\|_{L^1}+C'\int_1^{+\infty}\tau^{k-s/2-1}\tau^{-k}d\tau\leq C\|(-\Delta)^{k} \mathfrak{p}^\alpha_1\|_{L^1}+C''.
\end{align*}
It only remains to prove that $\|(-\Delta)^{k} \mathfrak{p}^\alpha_1\|_{L^1}<+\infty$, where $k$ is an integer. For this we use the estimates given in Theorem 7.3.2, page 320, of the book \cite{Kolokol}:
$$\left|\frac{\partial^m}{\partial x^m}\mathfrak{p}^\alpha_1(x)\right|\leq C\min\{1, |x|^{-m}\}\mathfrak{p}^\alpha_1(x), \text{ for } m=1,2,… $$
From this pointwise estimate we easily deduce that $\|(-\Delta)^{k} \mathfrak{p}^\alpha_1\|_{L^1}<+\infty$ and the proof of (\ref{FracEstimate1}) is now complete.

 The same ideas apply to the case $\alpha=2$ which is easier to handle as it corresponds with the usual heat kernel.


\quad\\[5mm]

\begin{multicols}{2}
\begin{minipage}[r]{90mm}
Diego \textsc{Chamorro}\\[3mm]
{\footnotesize
Universit\'e Paris-Saclay,\\
CNRS, Univ. Evry,\\ 
Laboratoire de Mod\'elisation\\ 
Math\'ematique d'Evry,\\ 
23 Boulevard de France,\\
91037 Evry, France\\[2mm]
diego.chamorro@univ-evry.fr
}
\end{minipage}
\begin{minipage}[r]{80mm}
Elena \textsc{Issoglio}\\[3mm]
{\footnotesize
School of Mathematics,\\  
University of Leeds,\\
Leeds, LS2 9JT\\
UK\\[2mm]
e.issoglio@leeds.ac.uk
}
\end{minipage}
\end{multicols}

\end{document}